\documentclass{article}
\usepackage{amsmath}
\usepackage{amssymb}

\newtheorem{theorem}{Theorem}

\newtheorem{corollary}[theorem]{Corollary}

\newtheorem{definition}[theorem]{Definition}

\newtheorem{proposition}[theorem]{Proposition}
\newtheorem{remark}[theorem]{Remark}

\newenvironment{proof}[1][Proof]{\noindent\textbf{#1.} }{\ \rule{0.5em}{0.5em}}
\input{tcilatex}
\begin{document}

\title{{\large The Ran-Reurings fixed point theorem without partial order: a
simple proof}}
\author{Hichem Ben-El-Mechaiekh\thanks{\textit{e-mail address: }%
hmechaie@brocku.ca} \\
Department of Mathematics and Statistics, Brock University\\
Saint Catharines, Ontario, Canada}
\date{ }
\maketitle

\textbf{Abstract.} The purpose of this note is to generalize the celebrated
Ran and Reurings fixed point theorem to the setting of a space with a binary
relation that is only transitive (and not necessarily a partial order) and a
relation-complete metric. The arguments presented here are simple and
straightforward. It is also shown that extensions by Rakotch and Hu-Kirk of
Edelstein's generalization of the Banach contraction principle to local
contractions on chainable complete metric spaces derive from the theorem of
Ran-Reurings.\bigskip 

\bigskip


\bigskip

\begin{center}
\textit{Dedicated to Professor Andrzej Granas}
\end{center}

\section{Preliminaries}

In 1961, M. Edelstein \cite{Ed} extended the Banach contraction principle by
establishing that every uniform local contraction $f:X\longrightarrow X$ of
an $\varepsilon $-chainable complete metric space $(X,d)$ has a unique fixed
point. In 1962, E. Rakotch \cite{R} refined Edelstein's result to a local
contraction $f$ of a complete metric space containing some rectifiable path
(i.e., a path of finite length\footnote{%
The \textit{length} of \ (continuous) path $\gamma :[0,1]\longrightarrow X$
is $l(\gamma ):=\sup \{L(P):P\in \mathcal{P}[0,1]\}$ where $\mathcal{P}[0,1]$
is the collection of all finite partitions $P=\{0=t_{0}<t_{1}<\ldots
<t_{n}=1\}$ of $[0,1]$, and $L(P)=\tsum\nolimits_{i=1}^{n}d(\gamma
(t_{i-1}),\gamma (t_{i})).$}) joining a given point $x_{0}$ to $f(x_{0}).$

Recall that a metric space $(X,d)$ is said to be $\varepsilon $\textit{%
-chainable} for some $\varepsilon >0,$\ if $\forall x,y\in X,\exists
\{u_{i}\}_{i=0}^{m}$ a\ finite sequence in $X$ such that: 
\begin{equation*}
x=u_{0},u_{m}=y\text{ and }d(u_{i-1},u_{i})<\varepsilon \text{ for all }%
i=1,\ldots ,m.
\end{equation*}

It is readily seen that a connected metric space is $\varepsilon $%
-chainable. Thus, if a metric space $X$ is rectifiably path-connected (i.e.,
any two points in $X$ are joined by a rectifiable path), then it is $%
\varepsilon $-chainable.

A mapping $f$ of a metric space $(X,d)$ onto itself is \textit{a} \textit{%
local contraction} (with constant $0<k<1)$ \textit{at a given point} $x\in
X, $ if there exists $\varepsilon _{x}>0$ such that $y,z\in B(x,\varepsilon
_{x})\Longrightarrow $ $d(f(y),f(z))<kd(y,z)$ \cite{R}. The mapping $f$ is a
local contraction on $X$ if it so at every point of $X.$

Improving on results of R. D. Holmes \cite{H}, Hu and Kirk \cite{HK}
established in 1978 a unique fixed point for a local radial contraction $f$
of a complete metric space containing an element $x_{0}$ joined to $f(x_{0})$
by a rectifiable path. Recall that a self-mapping $f$ of a metric space $%
(X,d)$ is said to be \textit{a local radial contraction at} $x$ if the
weaker condition $d(x,y)<\varepsilon _{x}\Longrightarrow
d(f(x),f(y))<kd(x,y) $ for $x,y$ in $X,$ holds \cite{H}. The mapping $f$ is
a local radial contraction on $X$ if it so at every point of $X.$

The authors in \cite{HK} insightfully noted that for a local radial
contraction $f:X\longrightarrow X,$ Rakotch's result readily reduces to the
Banach contraction principle applied to the restriction of $f$ on a
meaningful subspace of $X,$ namely the set $\tilde{X}$ consisting of those
points of $X$ that can be joined from $x_{0}$ by a rectifiable path. It
turns out that $\tilde{X}$ is kept invariant by $f$ and that $f$ is a
contraction for the \textit{path metric} $\tilde{d}(x,y)=\inf_{\gamma \in
\Gamma (x;y)}l(\gamma )$ on $\tilde{X},$ (where $\Gamma (x;y)$ is the
collection of all rectifiable paths joining $x$ to $y).$ As the completeness
of $(X,d)$ implies that of $(\tilde{X},\tilde{d})$, the Banach contraction
principle thus applies to $f$ on $(\tilde{X},\tilde{d}),$ yielding a fixed
point for $f.$

Before going any further let us recall A. C. M. Ran and C. B. Reurings
original result, which can be seen as a combination of the Banach
contraction principle and the Tarski's fixed point theorem (see e.g.,
Dugundli-Granas \cite{DG} for the two celebrated seminal results). A \textit{%
partial order} on a set $X$ is a binary relation $\preccurlyeq $ that is
reflexive, antisymmetric, and transitive; the pair $(X,\preccurlyeq )$
consisting of a set with a partial order is a\ \textit{poset.}

\begin{theorem}
(\cite{RR}, 2003) Let $(X,\preccurlyeq )$ be a poset where every pair $%
x,y\in X$ has an upper bound and a lower bound. Furthermore, let $d$ be a
metric on $X$ such that $(X,d)$ is a complete metric space. If $%
f:X\longrightarrow X$ is a continuous and monotonic (i.e., either
order-preserving or order-reversing) mapping such that:

(i) $\exists 0<k<1$ with 
\begin{equation*}
d(f(x),f(y))\leq kd(x,y),\ \forall \text{ }x\preceq y,
\end{equation*}

(ii) $\exists x_{0}\in X$ such that $x_{0}$ and $f(x_{0})$ are comparable%
\footnote{%
Two elements $x,y$ in a poset $(X,\preccurlyeq )$ are said to be \textit{%
comparable }if either $x\preccurlyeq y$ or $y\preccurlyeq x.$}.

Then $f$ has a unique fixed point $x^{\ast }\in X$ with $\lim_{n\rightarrow
\infty }f^{n}(x)=x^{\ast },$ $\forall x\in X.$
\end{theorem}

Nieto and Rodr\'{\i}guez-L\'{o}pez \cite{NR} noted in 2005 that the
continuity of the mapping $f$ in Theorem 1 can be replaced by the following
condition:

\textit{if a monotonic sequence} $\{x_{n}\}_{n\in 
\mathbb{N}
}\rightarrow x^{\ast }$ \textit{in} $X,$ \textit{then }$x_{n}$\textit{\ and }%
$x^{\ast }$\textit{\ are consistently comparable for all }$n\in 
\mathbb{N}
$ \textit{(i.e., }$x_{n}\preccurlyeq x^{\ast }$ \textit{for a non-decreasing
sequence).}

The aim of this paper is to extend the A. C. M. Ran and M. C. B. Reurings in 
\cite{RR}, but by considering a space $X$ equipped with a merely \textit{%
transitive} binary relation $\preccurlyeq $ (not a partial order) and with a
so-called relation-complete metric $d,$ and in which, suitable comparable
pairs can be joined by what we call $\varepsilon $-monotonic chains.

This is a significant departure from Ran-Reurings' theorem (Theorem 1) and
from its extensions by Nieto and Rodr\'{\i}guez-L\'{o}pez. Interestingly, we
also show that the theorem of Hu-Kirk can easily be derived.

Set-valued formulations of the results below are easily written and left to
the reader. Also, concrete motivations (whereby a partial order is not
available) can be constructed readily.

\section{Fixed point for a uniform local contraction on comparable elements}

In the remainder of this section, $(X,\preccurlyeq ,d)$ is a triple
consisting of a set $X$ together with a transitive binary relation $%
\preccurlyeq $ and a metric $d$ on $X.$ It should be kept in mind that the
relation $\preccurlyeq $ \textit{is not necessarily a partial order on }$X.$
Expediency imposes the occasional use of $X$ to designate $(X,\preccurlyeq
,d)$ in the absence of any confusion.

We introduce natural concepts of relation-chainability and
relation-completeness.

\begin{definition}
(i) Two elements $x,y$ in $X$ are said to be \textit{comparable} if either $%
x\preccurlyeq y$ or $y\preccurlyeq x.$

(ii) A mapping $f:X\longrightarrow X$ is said to be \textit{monotoni}c if it
is either always \textit{relation-preserving,} i.e., $x\preccurlyeq
y\Longrightarrow f(x)\preccurlyeq f(y)$ or always \textit{relation-reversing}%
, i.e., $x\preccurlyeq y\Longrightarrow f(y)\preccurlyeq f(x)$ for any given 
$x,y\in X.$

(iii) Analogously, a sequence $\{x_{n}\}_{n\in 
\mathbb{N}
}$ in $X$ is \textit{monotonic} if $x_{n}\preccurlyeq x_{n+1}$ for all $n$
or $x_{n+1}\preccurlyeq x_{n}$ for all $n.$

(iv) Two elements $x,y$ in $X$ are joined by an $\varepsilon $-monotonic
chain for some $\varepsilon >0$ if there exists a monotonic sequence $%
\{u_{i}\}_{i=0}^{m}$ in $X$ such that: 
\begin{equation*}
x=u_{0},u_{m}=y\text{ and }d(u_{i-1},u_{i})<\varepsilon \text{ for all }%
i=1,\ldots ,m.
\end{equation*}%
(Note that, by transitivity, $x$ and $y$ must be comparable.)

(v) The space $(X,\preccurlyeq ,d)$ is said to be $\varepsilon $-monotonic
chainable for some $\varepsilon >0,$ if any two comparable elements $x,y$ in 
$X$ are joined by an $\varepsilon $-monotonic chain.

(vi) The metric $d$ is monotonic complete if and only if every monotonic
Cauchy sequence converges in $X.$
\end{definition}

We start with fixed point results for a uniform local contraction on
comparable elements of $(X,\preccurlyeq ,d)$ where $d$ is a
relation-complete metric.

\begin{theorem}
Let $(X,\preccurlyeq ,d)$ be a\textit{\ triple consisting of a metric space }%
$(X,d)$ and a transitive binary relation $\preccurlyeq $ on $X,$ let $%
f:X\longrightarrow X$ be a mapping, and let $\varepsilon >0,$ be such that:

(a) $\exists x_{0}\in X$ such that $x_{0}$ and $f(x_{0})$ are joined by an $%
\varepsilon -$monotonic chain;

(b) $f$ is monotonic;

(c) if $\lim_{n\rightarrow \infty }f^{n}(x_{0})=x^{\ast }\in X,$ then $%
f^{n}(x_{0})$ and $x^{\ast }$ are comparable (consistent with the
monotonicity of $f$)$,\forall n;$

(d) $\exists 0<k<1$ such that for any comparable elements $x,y$ in $X,$ $%
d(x,y)<\varepsilon $ implies $d(f(x),f(y))\leq kd(x,y).$

Then, $f$ has a fixed point $x^{\ast }=\lim_{n\rightarrow \infty
}f^{n}(x_{0})$ provided the metric $d$ is monotonic complete.
\end{theorem}

\begin{proof}
By hypothesis, there exists a finite sequence $\{u_{i}\}_{0}^{m}$ with $%
d(u_{i-1},u_{i})<\varepsilon $ and, with no loss of generality,%
\begin{equation*}
x_{0}=u_{0}\preccurlyeq u_{1}\preccurlyeq \cdots \preccurlyeq
u_{m}=f(x_{0})\preccurlyeq f(u_{1})\preccurlyeq \cdots \preccurlyeq
f(u_{m})=f^{2}(x_{0})\preccurlyeq f^{2}(u_{1})\preccurlyeq \cdots
\end{equation*}

Thus,%
\begin{eqnarray*}
d(x_{0},f(x_{0})) &\leq
&\tsum\nolimits_{i=1}^{m}d(u_{i-1},u_{i})<m\varepsilon \\
d(f(x_{0}),f^{2}(x_{0})) &\leq
&\tsum\nolimits_{1}^{m}d(f(u_{i-1}),f(u_{i}))\leq
k\tsum\nolimits_{1}^{m}d(u_{i-1},u_{i})<mk\varepsilon \\
&&\ldots \\
d(f^{n}(x_{0}),f^{n+1}(x_{0})) &\leq
&k\tsum\nolimits_{1}^{m}d(f^{n-1}(u_{i-1}),f^{n}(u_{i}))\leq
k\tsum\nolimits_{1}^{m}d(u_{i-1},u_{i}) \\
&<&mk^{n}\varepsilon ,\text{ for all }n\in 
\mathbb{N}
.
\end{eqnarray*}

Surely, there exists $n_{0}\in 
\mathbb{N}
$ such that $0<mk^{n_{0}}<1.$ We show that the monotonic sequence $%
\{x_{n}=f^{n_{0}+n}(x_{0})\}_{n=1}^{\infty }$ is a Cauchy sequence in $X.$
Indeed, given $n^{\prime }>n,$%
\begin{eqnarray}
d(x_{n},x_{n^{\prime }}) &\leq &d(x_{n},x_{n+1})+\cdots +d(x_{n^{\prime
}-1},x_{n^{\prime }})  \notag \\
&\leq &k^{n_{0}}(k^{n}+k^{n+1}+\cdots +k^{n^{\prime }-1})m\varepsilon  \notag
\\
&=&k^{n}(1+k+\cdots +k^{n^{\prime }-n-1})\varepsilon  \notag \\
&=&k^{n}(\frac{1-k^{n^{\prime }-n}}{1-k})\varepsilon  \notag \\
&<&\frac{k^{n}}{1-k}\varepsilon .  \notag
\end{eqnarray}

Thus, $d(x_{n},x_{n^{\prime }})\rightarrow 0$ as $n\rightarrow \infty .$ By
monotonic completeness, the sequence $\{x_{n}\}_{1}^{\infty }$ converges to
some $x^{\ast }\in X$ which, by assumption (c), verifies $x_{n}\preccurlyeq
x^{\ast },\forall n.$

We conclude the proof by showing that $x^{\ast }=f(x^{\ast }).$

For any $\varepsilon ^{\prime }\in (0,\varepsilon ),$ there exists $%
n_{\varepsilon ^{\prime }}\in 
\mathbb{N}
$ such that $d(x_{n},x^{\ast })<\varepsilon /2$ for all $n\geq
n_{\varepsilon ^{\prime }}.$ For all $n>n_{\varepsilon ^{\prime }}$ and
since $x_{n}\preccurlyeq x^{\ast },$ it follows that $d(f(x^{\ast
}),f(x_{n-1}))\leq kd(x^{\ast },x_{n-1}).$ Now, as $x_{n}=f(x_{n-1}),$%
\begin{eqnarray*}
d(f(x^{\ast }),x^{\ast }) &\leq &d(f(x^{\ast }),f(x_{n-1}))+d(x_{n},x^{\ast
})\leq kd(x^{\ast },x_{n-1})+d(x_{n},x^{\ast }) \\
&<&k\frac{\varepsilon ^{\prime }}{2}+\frac{\varepsilon ^{\prime }}{2}%
<\varepsilon ^{\prime }.
\end{eqnarray*}%
As $0<\varepsilon ^{\prime }<\varepsilon $ is arbitrary, $f(x^{\ast
})=x^{\ast }=\lim_{n\rightarrow \infty
}f^{n_{0}+n}(x_{0})=\lim_{n\rightarrow \infty }f^{n}(x_{0}).$
\end{proof}

\begin{remark}
(1) Clearly, if the monotonic mapping $f:X\longrightarrow X$ globally
contracts comparable elements, i.e., 
\begin{equation*}
\exists 0<k<1\text{ with }d(f(x),f(y))\leq kd(x,y)\text{ for any comparable
pair }x,y\in X,
\end{equation*}%
and if there exists $x_{0}\in X$ comparable to $f(x_{0}),$ then, given any $%
\varepsilon >0,$ there exists $u_{0}=f^{n}(x_{0})$ with $n$ large, such that 
$u_{0}$ and $f(u_{0})=f^{n+1}(x_{0})$ are comparable, and $%
d(u_{0},f(u_{0}))<\varepsilon ,$ i.e., $u_{0}$ and $f(u_{0})$ are joined by
a two element $\varepsilon $-monotonic chain. Theorem 3 thus applies to the
pair $u_{0},f(u_{0})$ to immediately obtain Nieto-Rodr\'{\i}guez-L\'{o}pez's
version of Ran-Reurings' theorem. But here, we again point out that the
relation $\preccurlyeq $ is merely transitive and not an order relation.

(2) The existence of a fixed point holds if hypothesis (c) of Theorem 3 is
replaced by the less general assumption:

(c') $f$ is sequentially continuous along the sequence $f^{n}(x_{0})$, or
more generally if $f$ is monotonic-sequentially continuous, i.e., $%
f(\lim_{n\rightarrow \infty }x_{n})=\lim_{n\rightarrow \infty }f(x_{n})$ for
any monotonic converging sequence $\{x_{n}\}_{n\in 
\mathbb{N}
}$ in $X.$
\end{remark}

To secure uniqueness of the fixed point, we require global $\varepsilon $%
-monotonic chainability of the space as well as the existence, for any given
pair of elements $x,y\in X$, of a third element $z\in X$ similarly
comparable to both $x$ and $y$ (i.e., $z\preccurlyeq x$ and $z\preccurlyeq y$
or $x\preccurlyeq z$ and $y\preccurlyeq z).$

\begin{theorem}
If $(X,\preccurlyeq ,d),$ where $\preccurlyeq $ is a transitive relation and 
$d$ is a metric, is $\varepsilon $-monotonic chainable for some $\varepsilon
>0$ and $f:X\longrightarrow X$ is a mapping satisfying:

(a) $\exists x_{0}\in X$ such that $x_{0}$ and $f(x_{0})$ are comparable;

(b) $f$ is monotonic;

(c) if $\lim_{n\rightarrow \infty }f^{n}(x_{0})=x^{\ast }\in X,$ then $%
f^{n}(x_{0})$ and $x^{\ast }$ are comparable (consistent with the
monotonicity of $f$)$,\forall n;$

(d) $\exists 0<k<1$ such that if $x,y$ in $X$ are comparable, $%
d(x,y)<\varepsilon $ implies $d(f(x),f(y))\leq kd(x,y);$

(e) every pair of elements of $X$ admits a third element similarly
comparable to both.

Then, $f$ has a unique fixed point $x^{\ast }=\lim_{n\rightarrow \infty
}f^{n}(x)$ for any initial point $x\in X,$ provided the metric $d$ is
monotonic complete.
\end{theorem}

\begin{proof}
Proceeding along the lines of Ran-Reurings \cite{RR}, given an arbitrary
element $x\in X,$ we consider first the case where $x$ and $x_{0}$ are
comparable, say $x\preccurlyeq x_{0}.$ By hypothesis, $x$ and $x_{0}$ can be
joined by an $\varepsilon $-monotonic chain $x=v_{0}\preccurlyeq \cdots
\preccurlyeq v_{p}=x_{0}.$ Arguing as in the preceding proof, it is easy to
see that $d(f^{n}(v_{i-1}),f^{n}(v_{i}))\leq k^{n}\varepsilon $ for all $%
n=0,1,...$ and all $i=1,\ldots ,p.$ It follows that for any given $\delta
>0, $ there exists $n_{\delta }\in 
\mathbb{N}
,$ such that for $n\geq n_{\delta }:$%
\begin{equation*}
d(f^{n}(x),f^{n}(x_{0}))\leq k^{n}p\varepsilon <\frac{\delta }{2}\text{ and }%
d(f^{n}(x_{0}),x^{\ast })<\frac{\delta }{2}.
\end{equation*}%
Hence, $d(f^{n}(x),x^{\ast })\leq
d(f^{n}(x),f^{n}(x_{0}))+d(f^{n}(x_{0}),x^{\ast })<\frac{\delta }{2}+\frac{%
\delta }{2}=\delta ,$ i.e., $\lim_{n\rightarrow \infty
}f^{n}(x)=\lim_{n\rightarrow \infty }f^{n}(x_{0})=x^{\ast }.$

To complete the proof, let $x\in X$ be arbitrary and let $z\in X$ be
similarly comparable to both $x$ and $x_{0},$ say, 
\begin{equation*}
z\preccurlyeq x\text{ and }z\preccurlyeq x_{0}.
\end{equation*}

From the first part of the argument, we have:%
\begin{equation*}
\lim_{n\rightarrow \infty }f^{n}(z)=\lim_{n\rightarrow \infty
}f^{n}(x_{0})=x^{\ast }.
\end{equation*}

Also, $z$ and $x$ are joinable by an $\varepsilon $-monotonic chain, and as
above, for $n$ large enough, $d(f^{n}(z),f^{n}(x))$ can be made arbitrarily
small. Thus, $\lim_{n\rightarrow \infty }f^{n}(x)=\lim_{n\rightarrow \infty
}f^{n}(z)=x^{\ast }.$ This completes the proof.
\end{proof}

Quite interestingly, Rakotch \cite{R} and Hu-Kirk \cite{HK} theorems can be
obtained from Theorem 3 (in fact from Ran-Reurings' theorem). The proof
makes crucial use of the following key observations:

\begin{proposition}
\cite{HK} Let $f:X\longrightarrow X$ be a local radial contraction with
constant $0<k<1$ on a metric space $(X,d).$ Then $d(f(\gamma (0),f(\gamma
(1))\leq kl(\gamma )$ and $l(f(\gamma ))\leq kl(\gamma )$ for any
rectifiable path $\gamma :[0,1]\longrightarrow X.$
\end{proposition}

The reader is referred to \cite{HK} for the proof.

\begin{corollary}
\cite{HK} Let $(X,d)$ be a complete metric space and $f:X\longrightarrow X$
be a local radial contraction with constant $k\in (0,1)$. Suppose that there
exists $x_{0}\in X$ such that $x_{0}$ and $f(x_{0})$ are joined by a
rectifiable path. Then $f$ has a fixed point.
\end{corollary}

\begin{proof}
By hypothesis, there exists a rectifiable path $\gamma _{0}$ joining $x_{0}$
to $f(x_{0}).$ By Proposition 6 (i), each path $f^{n}(\gamma _{0})$ has
length smaller than $k^{n}l(\gamma _{0})$ and joins the element $%
f^{n}(x_{0}) $ to $f^{n+1}(x_{0}).$ Let $X_{0}=\{f^{n}(x_{0})\}_{n=0}^{%
\infty }$ (with $f^{0}(x_{0})=x_{0}).$ Obviously, $f(X_{0})=\{f^{n}(x_{0})%
\}_{n=1}^{\infty }\subset X_{0}.$ Define a total order on $X_{0}$ as
follows: $f^{n}(x_{0})\preccurlyeq f^{m}(x_{0})\Longleftrightarrow n\leq m.$
Clearly, $x_{0}\preccurlyeq f(x_{0})$ and $f$ is obviously monotonic on $%
X_{0}.$ Define a metric $d_{0}$ on $X_{0}$ as:%
\begin{gather*}
d_{0}(f^{n}(x_{0}),f^{m}(x_{0}))=d_{0}(f^{m}(x_{0}),f^{n}(x_{0}))=\tsum%
\nolimits_{i=n}^{m-1}l(f^{i}(\gamma _{0}))\text{ for }n<m, \\
d_{0}(x,y)=0\Leftrightarrow x=y=f^{n}(x_{0})\text{ for some }n\in
\{0,1,2,\ldots \}.
\end{gather*}

Note that the initial metric $d,$ the path metric$\footnote{%
Note that $X_{0}$ is a subset of the space $\tilde{X}:=\{x\in X;\Gamma
(x_{0};x)\neq \emptyset \}$ equipped with the path metric $\tilde{d}%
(x,y)=\inf_{\gamma \in \Gamma (x;y)}l(\gamma )$ defined in \cite{HK} and
mentioned in the Preliminaries above.}$ $\tilde{d},$ and the metric $d_{0}$
verify $d\leq \tilde{d}\leq d_{0}$ on $X_{0}.$

For any given pair $x,y\in X_{0},$ say $x=f^{n}(x_{0})$ and $y=f^{m}(x_{0})$
with $x\preccurlyeq y,$ it follows from Proposition 6 (i):%
\begin{eqnarray*}
d_{0}(f(x),f(y))
&=&d_{0}(f^{n+1}(x_{0}),f^{m+1}(x_{0}))=\tsum\nolimits_{i=n+1}^{m}l(f^{i}(%
\gamma _{0})) \\
&\leq &k\tsum\nolimits_{i=n}^{m-1}l(f^{i}(\gamma _{0}))=kd_{0}(x,y),
\end{eqnarray*}%
i.e., $f$ is a contraction on $X_{0}$ relative to the metric $d_{0}.$

Given an arbitrary but fixed $\varepsilon >0,$ one may assume, with no loss
of generality, that $d_{0}(f^{n}(x_{0}),f^{n+1}(x_{0}))<\varepsilon $ for $%
n=0,1,2,\ldots $ Indeed, since $k_{n}\downarrow 0^{+}$ as $n\rightarrow
\infty ,$ there exists a positive integer $n_{\varepsilon }$ large enough as
to have $k^{n}d_{0}(x_{0},f(x_{0}))<\varepsilon ,$ for all $n\geq
n_{\varepsilon }.$ One could then replace the full sequence of iterates $%
\{f^{n}(x_{0})\}_{n=0}^{\infty }$ by its tail $\{f^{n}(x_{0})\}_{n\geq
n\varepsilon }$ which verifies:%
\begin{equation*}
d_{0}(f^{n}(x_{0}),f^{n+1}(x_{0}))\leq
k^{n}d_{0}(x_{0},f(x_{0}))<\varepsilon \text{ for all }n\geq n_{\varepsilon
},
\end{equation*}%
and view $f^{n_{\varepsilon }}(x_{0})$ as the initial point in lieu of $%
x_{0}.$ Therefore, every two elements in $X_{0}$ can be joined by an $%
\varepsilon $-monotonic chain.

It was established in \cite{HK} that if the original metric $d$ is complete
on $X$ then the path metric $\tilde{d}$ is complete on the space $\tilde{X}$
of points joinable from $x_{0}$ by a rectifiable path. Naturally, the
closure of $X_{0}$ for the metric $d_{0}$ must also be complete. Indeed, let 
$\{x_{r}\}$ be a Cauchy sequence in $(X_{0},d_{0}).$ Since $d\leq d_{0},$
the sequence $\{x_{m}\}$ is also Cauchy in $(X,d),$ implying $%
\lim_{m\rightarrow \infty }d(x_{m},x^{\ast })=0$ for some $x^{\ast }\in X.$

We establish first that every element $x_{m}$ of the Cauchy sequence can be
joined to $x^{\ast }$ by a rectifiable path $\gamma .$ Indeed, let $%
\{\varepsilon _{i}\}$ be a sequence of summable positive real numbers, i.e., 
$\tsum\nolimits_{i=1}^{\infty }\varepsilon _{i}<\infty .$ For each $i,$
choose $m_{i}$ large enough so that $l(\gamma
_{i})=d_{0}(x_{m_{i}},x_{m_{i+1}})<\varepsilon _{i}$ where $\gamma _{i}$ is
the rectifiable path joining $x_{m_{i}}=f^{n_{m_{i}}}(x_{0})$ to $%
x_{m_{i+1}}=f^{n_{m_{i+1}}}(x_{0})$ consisting of finite union $%
f^{n_{m_{i}}}(\gamma _{0})\cup \ldots \cup f^{n_{m_{i}}-1}(\gamma _{0}).$
Each path $\gamma _{i}$ can be rescaled as a path $\gamma :[\frac{1}{i+1},%
\frac{1}{i}]\longrightarrow X.$ Define a path $\gamma :[0,1]\longrightarrow
X $ by putting $\gamma (t)=\gamma _{i}(t)$ for $t\in \lbrack \frac{1}{i+1},%
\frac{1}{i}]$ and $\gamma (0)=x^{\ast }.$ By construction, the path $\gamma $
is continuous on $(0,1].$ To ascertain continuity at $t=0,$ let $%
t_{k}\downarrow 0^{+}.$ Observe that each $t_{k}$ is in some interval $[%
\frac{1}{i+1},\frac{1}{i}]$ and, for all $k^{\prime }s$ large enough,%
\begin{eqnarray*}
d(\gamma (t_{k}),x^{\ast }) &\leq &d(\gamma
(t_{k}),x_{m_{i}})+d(x_{m_{i}},x^{\ast }) \\
&\leq &d(x_{m_{i}},x_{m_{i+1}})+d(x_{m_{i}},x^{\ast }) \\
&\leq &d_{0}(x_{m_{i}},x_{m_{i+1}})+d(x_{m_{i}},x^{\ast }) \\
&<&\varepsilon _{i}+d(x_{m_{i}},x^{\ast }).
\end{eqnarray*}

As $t_{k}\rightarrow 0,i\rightarrow \infty ,m_{i}\rightarrow \infty ,$ and $%
\varepsilon _{i}\rightarrow 0,$ thus $d(\gamma (t_{k}),x^{\ast })\rightarrow
0,$ i.e., $\gamma (t_{k})\rightarrow \gamma (0).$ It should be noted, in
addition, that the continuous path $\gamma $ joining $x^{\ast }$ and $%
x_{m_{1}}$ verifies $l(\gamma )\leq \tsum\nolimits_{i=1}^{\infty }l(\gamma
_{i})\leq \tsum\nolimits_{i=1}^{\infty }\varepsilon _{i}<\infty .$ Now,
define $d_{0}(x_{m_{i}},x^{\ast })=l(\gamma |_{[0,\frac{1}{i}]})$ and note
that 
\begin{equation*}
d_{0}(x_{m_{i}},x^{\ast })=l(\gamma |_{[0,\frac{1}{i}]})\leq
\tsum\nolimits_{j=i}^{\infty }\varepsilon _{j}\rightarrow _{i\rightarrow
\infty }0.
\end{equation*}

Since $\{x_{m_{i}}\}$ is a subsequence of the Cauchy sequence $\{x_{m}\}$ in 
$X_{0},$ it follows that $\lim_{m\rightarrow \infty }d_{0}(x_{m},x^{\ast
})=0,$ i.e., $x^{\ast }\in \overline{X_{0}}^{d_{0}},$ i.e., $\overline{X_{0}}%
^{d_{0}}$ is $d_{0}$-complete. Let us extend the binary relation $%
\preccurlyeq $ to $\overline{X_{0}}^{d_{0}}$ by putting:%
\begin{equation*}
\forall x\in X_{0},\forall z\in \overline{X_{0}}^{d_{0}}\setminus
X_{0},x\preccurlyeq z.
\end{equation*}

To conclude the proof, it remains to note that the mapping $f$ (a $d_{0}-$%
contraction on comparable elements of $X_{0})$ naturally extends to a
contraction on comparable elements of $\overline{X_{0}}^{d_{0}},$ thus
verifying all hypotheses of Theorem 3 on $\overline{X_{0}}^{d_{0}}.$
\end{proof}

\begin{remark}
Of course, it is much simpler to prove Corollary 7 by observing that $%
X_{0}=\{f^{n}(x_{0})\}_{n=0}^{\infty }$ is a Cauchy sequence in $(X,d)$,
hence convergent to a fixed point of $f.$ Indeed, $\forall m>n\geq
n_{\varepsilon },$ we do have%
\begin{eqnarray*}
d(f^{m}(x_{0}),f^{n}(x_{0})) &\leq
&\tsum\nolimits_{i=n}^{m-1}d(f^{i+1}(x_{0}),f^{i}(x_{0})) \\
&\leq &\tsum\nolimits_{i=n}^{m-1}k^{i+1}l(\gamma _{0}) \\
&=&k^{n+1}(1+\cdots +k^{m-(n+1)})l(\gamma _{0}) \\
&=&k^{n+1}(\frac{1-k^{m-n}}{1-k})l(\gamma _{0}) \\
&<&\frac{k^{n}}{1-k}l(\gamma _{0})\rightarrow _{n\rightarrow \infty }0.
\end{eqnarray*}%
But our point\ here is to show that the results of Hu-Kirk and Rakotch
follow also from Theorem 3 (and indeed from the Ran-Reurings
theorem).\medskip
\end{remark}

\textbf{Acknowledgment:} The author is indebted to Marlene Frigon and
Mohamed A. Khamsi for bringing to his attention the remarkable paper of
Jacek Jachymski \cite{J} where the Ran-Reurings fixed point theorem is
significantly extended to complete metric spaces endowed with a directed
graph, as well as to reference \cite{AK} for the extension of Caristi's
fixed point theorem to such spaces. The use of the language and concepts
from graph theory in \cite{J} allows for the unification of the main results
in \cite{Ed}, \cite{HK}, \cite{NR}, \cite{R}, and \cite{RR} and for the
consideration of a quasi-order (a reflexive and transitive relation) instead
of a partial order. Theorem 3 is to be compared with Theorem 3.4 in \cite{J}%
; it does not require the completeness of the metric over the whole space $X$
but rather over chains. In addition, while keeping the Ran and Reurings'
perspective of a compatibility between a metric and a (merely transitive)
binary relation, the argument used here are simple and straightforward.

The author also thanks Dr. Asma Rashid Butt for sparking his interest in
these aspects of metric fixed point theory.

\end{document}